\documentclass{article}       
\usepackage{graphicx}
\usepackage{todonotes}
\usepackage[utf8x]{inputenc}
\usepackage[T1]{fontenc}
\usepackage{lmodern}
\usepackage[english]{babel}
\usepackage{enumitem}
\usepackage{microtype}
\usepackage{textcomp}
\usepackage{amsthm, amsmath,amsfonts,amssymb,mathtools}
\usepackage{trfsigns}
\usepackage{units}
\usepackage{hyperref}
\newtheorem{thm}{Theorem}
\numberwithin{thm}{section}
\newtheorem{lem}[thm]{Lemma}
\newtheorem{cor}[thm]{Corollary}
\newtheorem{prop}[thm]{Proposition}

\newtheorem*{acknowledgements}{Acknowledgements}

\newtheorem{defi}[thm]{Definition}

\newtheorem{rem}[thm]{Remark}

\newcommand{\R}{\mathbb{R}}
\newcommand{\Rbar}{\overline{\mathbb{R}}}

\newcommand{\X}{\mathcal{X}}

\newcommand{\norm}[1]{\left\|#1\right\|}
\newcommand{\paren}[1]{\left(#1\right)}
\newcommand{\abs}[1]{\left|#1\right|}
\newcommand{\lsp}{\left\langle}
\newcommand{\rsp}{\right\rangle}

\newcommand{\f}{f} 

\newcommand{\p}{\xi}	

\newcommand{\taubar}{\overline{\tau}}

\newcommand{\F}{\mathcal{F}}

\newcommand{\conv}{r} 
\newcommand{\smoo}{s} 
\newcommand{\breg}[2]{\Delta_{#1}^{#2}}
\newcommand{\symbreg}[1]{\Delta_{#1}^{\rm{sym}}}

\newcommand{\keywords}[1]{\ \\ {\bf Keywords:} #1 \\ {\bf MSC:} 46N10, 47N10 }

\title{Upper and lower bounds for the Bregman divergence}

\author{Benjamin Sprung\thanks{Lotzestr. 16-18, 37083 G\"ottingen, Germany, \href{mailto:b.sprung@math.uni-goettingen.de}{b.sprung@math.uni-goettingen.de}}        
}

\begin{document}

\maketitle

\begin{abstract}
In this paper we study upper and lower bounds on the Bregman divergence $\breg{\F}{\p}(y,x):=\F(y)-\F(x)-\lsp \p, y-x\rsp $ for some convex functional $\F$ on a normed space $\X$, with subgradient $\p\in\partial\F(x)$. We give a considerably simpler new proof of the inequalities by  
Xu and Roach for the special case $\F(x)=\norm{x}^p, p>1$. The results can be transfered to more general functions as well.


\end{abstract}

\section{Introduction}

In recent times the Bregman divergence (or Bregman distance) $\breg{\F}{x^*}(y,x)$, introduced by Bregman in \cite{Bregman1967}, has been used as a generalized distance measure in various branches of applied mathematics, for example optimization, inverse problems, statistics and computational mathematics, especially machine learning. To get an overview over the Bregman divergence and its possible applications in optimization and inverse problems we refer to \cite{Censor97} respectively \cite{Burger16}. In particular the Bregman divergence has been used for various algorithms in numerical analysis and also for convergence analysis of numerical methods and algorithms. 

Especially when doing convergence analysis it is often crucial to have lower and upper bounds on the Bregman divergence in terms of norms. 
In \cite{Xu1991} the authors prove 
upper and lower bounds for expressions 
\begin{align}\label{eq:lin_err}
	\norm{x+y}^p-\norm{x}^p-p\lsp j_p (x),y\rsp=:\breg{\F}{j_p(x)}(x+y,x),
\end{align}
where $j_p:\X\to\X^*$ is a duality mapping, under certain assumptions on the Banach space $\X$. As it turns out that \eqref{eq:lin_err} is the Bregman divergence corresponding to the functional $\F=\norm{\cdot}^p$ these results have been used since then in many papers working with the Bregman divergence. However from the proofs of \cite{Xu1991} it seems difficult to transfer the results to other functions $\F$. Thus we develop in this work a simple framework to find such bounds and in fact can apply it to give a short new  proof of the results from \cite{Xu1991} for $\F(x)=\norm{x}^p, p>1$ . 

Our approach is as follows: Proving upper bounds is rather simple if one sufficiently understands the smoothness of $\F$ as the Bregman divergence is  basically a linearization error and 
linearization errors are related to differentiability by definition. In particular we will show that one can obtain upper bounds for the Bregman divergence corresponding to $\F=\phi(\norm{\cdot})$, if $\phi:\R\to\R$ is convex and sufficiently smooth.

Regarding lower bounds we will make use of $\F^*$, the convex conjugate of $\F$. Actually it can be shown that lower bounds for $\breg{\F}{x^*}(y,x)$ correspond to upper bounds for 
$\breg{\F^*}{x}(y^*,x^*)$. Note that this idea is not at all new. Already in \cite{Zalinescu83} this kind of connection between $\F$ and $\F^*$ was discussed in depth. So again one can just make use of the smoothness of 
$\F^*$ to conclude lower bounds for $\breg{\F}{x^*}(y,x)$. One might argue that convex conjugates can be rather complicated functions and expecting differentiability is too optimistic.  This is true to some extent, but actually reasonable lower bounds on $\breg{\F}{x^*}(y,x)$ already imply differentiability of $\F^*$ at $x^*$ 
(see \cite[Theorem 2.1]{Zalinescu83}). So if one has any hope on finding lower bounds then one might as well work with the convex conjugate. 

One reason why our proof is simpler than the proof from \cite{Xu1991} is that they did it the other way round. They firstly proved lower bounds with quite some effort and then used the convex conjugate to show upper bounds.

We will focus mainly on asymptotic bounds for $\breg{\F}{x^*}(y,x)$ as $\norm{x-y}\to 0$.  It is the more interesting case for applications as for example in convergence analysis one
will be interested in the Bregman divergence of $x_n$ and $x$, where $x_n\to x$. Also theoretical it is the more challenging case, since for $\norm{x-y}\to\infty$ the Bregman divergence $\breg{\F}{x^*}(y,x)$ will mostly depend on the behavior of $\F(y)$ as $y$ tends to infinity and it should be easy to find lower and upper bounds. In particular we will show at the end of the paper, how one can deduce uniform bounds for all $x,y\in\X$ from the asymptotic bounds for the case $\F=\norm{\cdot}^p$.

The paper consists of 4 sections. In Section \ref{chap:conv_anal} we recall some basis notions of convex analysis. In Section \ref{chap:moduli} we define moduli of smoothness and convexity corresponding to a general functional $\F$ and develop some properties of them. Finally in Section \ref{chap:norm_powers} we then use the theory from Section \ref{chap:moduli} on the functional $\F=\tfrac{1}{p}\norm{\cdot}^p$ for $p>1$ and find lower and upper bounds for the corresponding Bregman divergence given by the smoothness respectively the convexity of the space $\X$ as shown in \cite{Xu1991}.

\section{Tools from convex analysis} \label{chap:conv_anal}

In this work $\X$ will always be a real Banach space, with $\dim \X\ge 2$, $\X^*$ denotes its dual space, $S_\X=\{x\in\X:\norm{x}=1\}$ the unit sphere and $\F: \X\to \Rbar:=\R\cup \{\infty\}$ some function.   We will need some basic concepts from convex analysis, so we shortly recall them in this chapter.

	$x^*\in\X^*$ is called a \emph{subgradient} of a convex function $\F: \X\to \Rbar$ at $x\in\X$ if $\F(x)$ is finite and
	\begin{align} \label{eq:subgradient}
		\F(y)\ge \F(x) +\lsp x^*, y-x \rsp,
	\end{align}
	for all $y\in \X$. The set of all subgradients of $\F$ at $x$ is called the \emph{subdifferential} of $\F$ at $x$ and denoted by $\partial\F(x)$.
	The \emph{convex conjugate} $\F^*:\X^* \to \Rbar$ of $\F$ is defined by
	\begin{align*}
		\F^*(x^*)=\sup_{x\in\X}\left[ \lsp x^*,x\rsp-\F(x) \right].
	\end{align*}
From this two definitions one can directly conclude the following \emph{generalized Young (in)equality}.
	For all $x\in\X, x^*\in\X^*$ we have
	\begin{align} \label{eq:young}
		\F(x)+\F^*(x^*)\ge \lsp x^*,x\rsp.
	\end{align}
	Equality holds true if and only if $x^*\in\partial\F(x)$. Further we have
\begin{align} \label{eq:doublestar}
	\F\ge \F^{**}:=\paren{\F^*}^*,
\end{align}
where equality holds if and only if $\F$ is convex and lower-semicontinuous.
Finally we define the object of interest of this work.
	For $\F(x)<\infty$ and $x^*\in\partial\F(x)$ the \emph{Bregman divergence} $\breg{\F}{x^*}(y,x)$ is given by
	\begin{align*} 
		\breg{\F}{x^*}(y,x)=\F(y)-\F(x)-\lsp x^*, y-x\rsp\ge 0,
	\end{align*}
	for all $y\in\X$.
We will be especially interested in functionals $\F(x)=\tfrac{1}{p}\norm{x}^p$ for some $p\ge 1$ and need to understand their subdifferentials, so finally we have the following.
	 For some $p\ge 1$ the set-valued mapping $J_p:\X\to 2^{X^*}$ given by
	 \begin{align*}
	 	J_p(x)=\left\{ x^*\in\X^*: \lsp x^*,x\rsp =\norm{x^*}\norm{x}, \norm{x^*}=\norm{x}^{p-1} \right\}
	 \end{align*}
	 is called the \emph{duality mapping} with respect to $p$ of $\X$. The sets $J_p(x)$ are always non-empty. A mapping $j_p:\X\to\X^*$ is called \emph{selection} of $J_p$ if 
	 $j_p(x)\in J_p(x)$ for all $x\in\X$.
	 If $\F(x)=\tfrac{1}{p}\norm{x}^p$, then we have \cite[Chap.1, Theorem 4.4]{Cioranescu1990}
	 \begin{align*}
	 	\partial\F(x)=J_p(x).
	 \end{align*}

\section{Moduli of smoothness and convexity} \label{chap:moduli}

Finding upper bounds for \eqref{eq:lin_err} is related to the smoothness of the norm of $\X$ whereas lower bounds are related to convexity. Thus it is necessary to understand the moduli of smoothness and convexity of the space $\X$ and we shortly
 recall their definitions (see e.g.~\cite{Lindenstrauss1979}): 
\begin{defi}\label{defi:smooth_convex}
	The modulus of convexity $\delta_\X\colon [0,2]\to [0,1]$ of the space $\X$ is defined by
	\begin{equation*}
		\delta_\X(\varepsilon):=\inf \{1-\norm{y+\tilde{y}}/2:y,\tilde{y}\in S_\X, \norm{y-\tilde{y}}=\varepsilon\}.
	\end{equation*}
	The modulus of smoothness $\rho_\X\colon [0,\infty)\to [0,\infty)$ of $\X$ is defined by
	\begin{equation*}
		\rho_\X(\tau):= \sup\{(\norm{x+\tau y}+\norm{x-\tau y})/2-1:x,y\in S_\X\}.
	\end{equation*}
	The space $\X$ is called \textit{uniformly convex} if $\delta_\X(\varepsilon)>0$ for every $\varepsilon>0$. It is called \textit{uniformly smooth} 
	if $\lim_{\tau\to 0} \rho_\X(\tau)/\tau=0.$
	The space $\X$ is called \textit{$\conv$-convex} (or convex of power type $\conv$) if there exists a constant $K>0$ such that $\delta_\X(\varepsilon)\ge K\varepsilon^\conv$ for all $\varepsilon\in [0,2]$. Similarly, it is called $\smoo$-smooth (or smooth of power type $\smoo$) if $\rho_\X(\tau)\le K\tau^\smoo$ for all $\tau>0$.
\end{defi}
These two moduli have a well-developed theory, which is known in the literature for a long time and we will not discuss all their properties. However for our proofs we will need some specific properties stated in the following.
\begin{lem} \label{lem:mods_properties}
	\begin{enumerate}
		\item We have for $\tau_1\le\tau_2 $ that $\rho_\X(\tau_1) /\tau_1\le \rho_\X(\tau_2)/\tau_2$. \label{lem:mods_properties1}
		\item We have for all $\taubar>0$ that there exists a constant $C_{\taubar}$ such that for all Banach spaces $\X$ we have \label{lem:mods_properties4}
		\begin{align*}
			\rho_\X(\tau)\ge (1+\tau)^\frac{1}{2}-1\ge C_{\taubar}\tau^2, \qquad \tau\le\taubar.
		\end{align*}
		\item If $\delta_\X$ is extended by $\infty$ on $\R\setminus [0,2]$ then $(2\delta_\X)^*=2\rho_{\X^*}$. \label{lem:mods_properties2}
		\item There exists a convex function $f$ such that $\delta_\X(\tau/2)\le f(\tau)\le \delta_\X(\tau)$. In particular we have \label{lem:mods_properties3}
		$ \delta_\X^{**}(\tau)\ge  \delta_\X(\tau/2)$.
	\end{enumerate}
\end{lem}
\begin{proof}
All statements follow easily from \cite[Ch. 1.e]{Lindenstrauss1979}.
\end{proof}
For our purposes it will be more natural to introduce new definitions of the moduli of smoothness and convexity related to functionals instead of spaces. 
\begin{defi}\label{defi:Fsmooth_Fconvex}
	Let $\F\colon \X\to\Rbar$ be some arbitrary function, $x\in \X$, $\F(x)<\infty$ and $\p\in\X^*$. 
	Define the \emph{linearization error functional}  $\breg{\F}{\p}(y,x)$ by
	\begin{align*} 
		\breg{\F}{\p}(y,x)=\F(y)-\F(x)-\lsp \p, y-x\rsp.
	\end{align*}
	The modulus of smoothness $\rho_{\F,x}^\p\colon [0,\infty)\to [0,\infty]$ of $\F$ in $x$ with respect to $\p$ is defined by
	\begin{align*}
		\rho_{\F,x}^\p(\tau):=\sup_{y\in S_\X}\abs{\F(x+\tau y)-\F(x)-\lsp \p,\tau y\rsp}=\sup_{\norm{x-y}=\tau}\abs{\breg{\F}{\p}(y,x)}.
	\end{align*}
	The modulus of convexity $\delta_{\F,x}^\p\colon [0,\infty)\to [0,\infty]$ of $\F$ in $x$ with respect to $\p$ is defined by
	\begin{align*}
		\delta_{\F,x}^\p(\tau):=\inf_{\norm{x-y}=\tau}\abs{\breg{\F}{\p}(y,x)}.
	\end{align*}
	$\F$ is called \textit{$\conv$-convex} (or convex of power type $\conv$) in $x$ (w.r.t. $\p$) if there exists $K,\taubar>0$ such that $\delta_{\F,x}^\p(\tau)\ge K\tau^\conv$ for all $0<\tau\le \taubar$. Similarly, it is called $\smoo$-smooth (or smooth of power type $\smoo$) in $x$ (w.r.t. $\p$) if $\rho_{\F,x}^\p(\tau)\le K\tau^\smoo$ for all $0<\tau\le \taubar$.
\end{defi}

The quantities $\rho_{\F,x}^\p$, $\delta_{\F,x}^\p$  give us a reformulation of our basic problem: We want to find upper bounds for $\rho_{\F,x}^\p(\tau)$ and lower bounds for 
$\delta_{\F,x}^\p(\tau)$. Before we show some properties of these functions we should state some simple facts for their interpretation.
\begin{rem}
	We will mostly consider convex functions $\F$ with $\p\in\partial\F(x)$ so that the linearization error functional is a Bregman divergence and one can neglect the absolute value. 
	
	$\F$ is Fr\'{e}chet-differentiable in $x$ if and only if there exists  $\p\in\X^*$, such that $\rho_{\F,x}^\p(\tau)/\tau\to 0$ as $\tau\to 0$. $\F$ being $\smoo$-smooth in $x$, with $\smoo\in (1,2]$ then can be seen as a stronger form of differentiability, comparable to fractional derivatives, however $\F$ being $2$-smooth is not equivalent to twice differentiability but rather to the notion of strong smoothness.
	
	If there exists a selection $j:\X\to \X^*$ of the subdifferential of $\F$, i.e. for every $x$ exists $j(x)\in\partial\F(x)$, then this implies already that $\F$ is convex. 
	$\delta_{\F,x}^{j(x)}(\tau)>0$ for all $x,\tau$ implies strict convexity and as before $\conv$-convexity is an even stronger notion of convexity and $2$-convexity is connected to strong convexity. In \cite{Butnariu97} the \emph{modulus of local (or total) convexity of $\F$}, $\nu_\F(x,\tau)$, was introduced and is basically given by $\delta_{\F,x}^\p(\tau)$ just that 	$\lsp \p,y-x\rsp$ is replaced by the right hand side derivative of $\F$ at $x$ in direction $y-x$. If $\F(x)$ is convex and G\^{a}teaux-differentiable then $\nu_\F(x,\tau)$ coincides with $\delta_{\F,x}^{\p}(\tau)$, where $\p=\F'(x)$. The modulus of total convexity has been studied in several papers.
\end{rem}

It turns out that for functionals $\F$ that originate from the norm of $\X$ the moduli of the space and of the functions are closely related.

\begin{prop} \label{prop:smoo_mods_equiv}
	Let $\F=\norm{\cdot}_\X$ and  for all $x\in\X$ let $\p_x\in \partial\F(x)$ be arbitrary.
	We have
		\begin{align}
			 \rho\le  \sup_{x\in S_\X}\rho_{\F,x}^{\p_x}\le 2\rho.
		\end{align}	
\end{prop}

\begin{proof}
	We have
	\begin{align*}
		2\sup_{x\in S_\X}\rho_{\F,x}^{\p_x}(\tau)\ge \sup \left\{ \F(x+\tau y)+\F(x-\tau y)-2 : x,y\in S_\X \right\}=2 \rho(\tau).
	\end{align*}
	and for all $x,y\in S_\X$ we have by the definition of the subdifferential that
	\begin{align*}
		\F(x+\tau y)-\F(x)-\lsp\p_x,\tau y\rsp\le \F(x+\tau y)+\F(x-\tau y)-2\le 2 \rho(\tau).
	\end{align*}
\end{proof}

So this already gives us an upper bound for $\rho_{\norm{\cdot}_\X,x}^\p(\tau)$ if $x\in S_\X, \p\in\partial\F(x)$. For generalizing this to all $x\in\X$ we use the following.

\begin{prop}\label{prop:homo}
	If the functional $\F$ is positively $q$-homogeneous then we have for all $x\in\X,\p\in\X^*$ that
	\begin{align*}
		\norm{x}^q\delta_{\F,x/\norm{x}}^{\p/\norm{x}^{q-1}}\paren{\frac{\norm{x-y}}{\norm{x}}}\le \abs{\breg{\F}{\p}(y,x)}\le \norm{x}^q\rho_{\F,x/\norm{x}}^{\p/\norm{x}^{q-1}} \paren{\frac{\norm{x-y}}{\norm{x}}}
	\end{align*}
	and $\p/\norm{x}^{q-1}\in\partial\F(x/\norm{x})$ if and only if $\p\in\partial\F(x)$.
\end{prop}
\begin{proof}
	If $\F$ is positively $q$-homogeneous we have 
	\begin{align*}
		\abs{\breg{\F}{\p}(y,x)}=\norm{x}^q\abs{\breg{\F}{\p/\norm{x}^{q-1}}\paren{\frac{y}{\norm{x}},\frac{x}{\norm{x}}}},
	\end{align*}
	so that  the first claim follows from Definition \ref{defi:Fsmooth_Fconvex} . The second claim follows from multiplying \eqref{eq:subgradient} either by $\norm{x}^{q}$ or $\norm{x}^{-q}$.
\end{proof}
For convex functions $\F$ one can show that both moduli are nondecreasing.
\begin{prop} \label{prop:mods_nondecr}
		 Let $\F$ be convex, $x\in\X$ and $ \p\in\partial\F(x)$. Then for $\lambda\ge 1$ one has 
		\begin{align*}
			\rho_{\F,x}^\p(\lambda\tau)\ge \lambda \rho_{\F,x}^\p(\tau), \qquad 
			\delta_{\F,x}^\p(\lambda\tau)\ge \lambda \delta_{\F,x}^\p(\tau). 
		\end{align*}		
		 In particular $\delta_{\F,x}^\p, \rho_{\F,x}^\p$ are nondecreasing. 		
\end{prop}

\begin{proof}
			 The idea is the same, as in \cite{Butnariu97}. Let $\lambda\ge 1$. For all $y\in\X, \norm{y-x}=\tau$ one can define 
			$y_\lambda=\lambda y+(1-\lambda)x$, so $ \norm{y_\lambda-x}=\lambda\tau.$ 
			Then by convexity of $\F$ we get
			\begin{align*}
				\frac{1}{\lambda}\breg{\F}{\p}(y_\lambda,x)=\frac{1}{\lambda}\big(\F(\lambda y+(1-\lambda)x)-\F(x)\big)-\lsp \p,y-x\rsp\ge \breg{\F}{\p}(y,x) .
			\end{align*}
			So for all $y\in\X, \norm{y-x}=\tau$ we find
			\begin{align*}
				\breg{\F}{\p}(y,x)\le \frac{1}{\lambda} \rho_{\F,x}^\p(\lambda\tau),
			\end{align*}
			which gives the first inequality.
			Similarly for all $y\in\X, \norm{y-x}=\lambda\tau$ one can define $\tilde{y}_\lambda =\tfrac{1}{\lambda}y+(1-\tfrac{1}{\lambda})x$, then $ \norm{\tilde{y}_\lambda-x}=\tau$  
			and again convexity of $\F$ can be used to show $\breg{\F}{\p}(y,x)\ge \lambda \breg{\F}{\p}(\tilde{y}_\lambda,x)$, which yields the other inequality.	
\end{proof}

We also have a chain rule. 
\begin{prop}\label{prop:chain_rule}
	Let $\f\colon \R\to\R$ and $x\in\X,\p\in\X^*,t\in\R$ be such that $\rho_{\f,\F(x)}^{t}$ is nondecreasing. Then for all $\tau\ge 0$ we have
	\begin{align*}
		\rho_{\f\circ\F,x}^{t\p}(\tau)\le \abs{t}\rho_{\F,x}^{\p}(\tau)+\rho_{\f,\F(x)}^{t}\paren{\norm{\p}\tau+\rho_{\F,x}^{\p}(\tau)}.
	\end{align*}
\end{prop}

\begin{proof}
	Let $s=\F(x)$ and define functions $R,r$ by
	\begin{align*}
		\F(x+y)-\F(x)&=\lsp\p,y\rsp+R(y) && \forall y\in\X \\
		\f(s+h)-\f(s)&=th+r(h) && \forall h\in\R.
	\end{align*}
	Then we have for $\tau>0$ and $y\in S_\X$ that
	\begin{align*}
		\f\circ\F(x+\tau y)-\f\circ\F(x)&=t\paren{\lsp \p,\tau y \rsp+R(y)}+r\paren{\lsp \p,\tau y \rsp+R(\tau y)} \\
		&=\lsp t\p,\tau y \rsp+tR(\tau y)+r\paren{\lsp \p,\tau y \rsp+R(\tau y)}.
	\end{align*}
	Now the claim follows from $R(\tau y)\le\rho_{\F,x}^{\p}(\tau)$ and $r(h)\le\rho_{\f,\F(x)}^{t}(\abs{h})$ together with the assumption that
	$\rho_{\f,\F(x)}^{t}$ is a nondecreasing function.
\end{proof}
Propositions \ref{prop:smoo_mods_equiv}, \ref{prop:mods_nondecr} and  \ref{prop:chain_rule} are already sufficient to find upper bounds on $\rho_{\F,x}^{\p}$ for $\F=\f\paren{\norm{x}_\X}$ if $f$ is convex and we sufficiently understand the smoothness of $\f$ and of the space $\X$. Regarding lower bounds the following proposition will be our key instrument.
\begin{prop}\label{prop:conv*=smoo}
	Let $\F$ convex and $x$ be such there exists $\p\in\partial\F(x)$. We have
	\begin{align} \label{eq:conv*=smoo}
		\paren{\delta_{\F,x}^\p}^*=\rho_{\F^*,\p}^x .
	\end{align}
	 Further we have that $\F$ is $p$-convex in $x$ w.r.t. $\p$ if and only if $\F^*$ is $p'$-smooth in $\p$ w.r.t. $x$.
\end{prop}

\begin{proof}
We have
\begin{align*}
	\rho_{\F^*,\p}^x(\tau)&=\sup_{y^*\in S_{\X^*}}\left[\F^*(\p+\tau y^*)-\F^*(\p)-\lsp \tau y^*,x\rsp\right] \\
	&=\sup_{y^*\in S_{\X^*}}\sup_{y\in\X}\left[\lsp \p+\tau y^*,y \rsp-\F(y)-\F^*(\p)-\lsp \tau y^*,x\rsp\right] \\
	&=\sup_{y\in\X}\left[\lsp \p,y \rsp-\F(y)-\F^*(\p)+\tau\norm{y-x}\right].
\end{align*}	
By Youngs equality \eqref{eq:young} we then have
\begin{align*}
	\rho_{\F^*,\p}^x(\tau)
	&=\sup_{y\in\X}\left[\F(x)-\F(y) +\lsp \p,y-x \rsp+\tau\norm{y-x}\right] \\
	&=\sup_{\varepsilon\in\R_0^+}\sup_{y\in\X,\norm{y-x}=\varepsilon}\left[\varepsilon\tau-\breg{\F}{\p}(y,x)\right]=\paren{\delta_{\F,x}^\p}^*(\tau).
\end{align*}	
The second statement follows from \eqref{eq:conv*=smoo}, which gives that 
			\begin{align*}
				\rho_{\F^*,\p}^x=\paren{\delta_{\F,x}^\p}^*, \qquad \delta_{\F,x}^\p\ge \paren{\delta_{\F,x}^\p}^{**}=\paren{\rho_{\F^*,\p}^x}^*
			\end{align*}			 
			and the fact that by Proposition \ref{prop:mods_nondecr} we have for $\tau>\taubar$ that 
			$\rho_{\F,x}^\p(\tau)\ge \tau \rho_{\F,x}^\p(\taubar)/\taubar, \delta_{\F,x}^\p(\tau)\ge \tau \delta_{\F,x}^\p(\taubar)/\taubar$, so that in particular
			\begin{align*}
				\paren{\delta_{\F,x}^\p}^*(\tau^*)=\sup_{0\le\tau\le\taubar}\left[ \tau^*\tau-\delta_{\F,x}^\p(\tau)\right], \text{ for } \tau^*\le  \delta_{\F,x}^\p(\taubar)/\taubar \\
				\paren{\rho_{\F,x}^\p}^*(\tau^*)=\sup_{0\le\tau\le\taubar}\left[ \tau^*\tau-\rho_{\F,x}^\p(\tau)\right], \text{ for }  \tau^*\le \rho_{\F,x}^\p(\taubar)/\taubar.
			\end{align*}
			Thus one can just put in the corresponding lower or  upper bound and  calculate the maximum, which completes the proof.
\end{proof}

\section{Application to norm powers} \label{chap:norm_powers}

In this section we will consider $\F=\frac{1}{p}\norm{\cdot}^p$ for some $p>1$ and use the theory from the last chapter to reproduce the main results from \cite{Xu1991}. Note that in light of
Proposition \ref{prop:homo} it is sufficient to understand $\delta_{\F,x}^{j_p(x)}$ and $\rho_{\F,x}^{j_p(x)}$ for $x\in S_\X$.
\begin{thm} \label{thm:asy_xu_roach}
	For some fixed $p>1$ let $\F=\frac{1}{p}\norm{\cdot}^p$.
	\begin{enumerate}
		\item For all $\taubar>0$ exists a constant $C_{\taubar,p}>0$, such that for $x\in S_\X$ and $\tau\le \taubar$ we have \label{thm:asy_xu_roach1}
		\begin{align*} 
			\rho_{\F,x}^{j_p(x)}(\tau)\le C_{\taubar,p}\rho_{\X} (\tau)
		\end{align*} 
		\item If we have for $\taubar>0, \tau\le\taubar$ and all $x\in S_\X$  that \label{thm:asy_xu_roach2}
		\begin{align*}
			\rho_{\F,x}^{j_p(x)}(\tau)\le \phi(\tau),
		\end{align*}
		then
		\begin{align*}
			\rho_{\X}(\tau)\le p^{1/p-1}\phi(\tau)+C_{\taubar}\tau^2,
		\end{align*}
		for $\tau\le\taubar$. In particular if $\phi\colon \R^+\to\R^+$ fulfills $\lim_{\tau\to 0}\phi(\tau)/\tau=0$, then $\X$ is uniformly smooth.
		\item 
				Let $\tfrac{1}{p}+\tfrac{1}{p'}=1$. For all $x\in S_\X, \taubar>0$  we have  \label{thm:asy_xu_roach3}
							\begin{align*}
									\delta_{\F,x}^{j_p(x)}(\tau)\ge C_{\taubar,p'} \delta_{\X} (\tau/C_{\taubar,p'}), \qquad \tau \le C_{\taubar,p'}\rho_{\X^*}(\taubar)/\taubar
							\end{align*}
							where $C_{\taubar,p'}$ is the constant from \ref{thm:asy_xu_roach1}. and $\rho_{\X^*}(\taubar)/\taubar>0$.
		\item	\label{thm:asy_xu_roach5}
			If there exists $\taubar>0$ such that we have for all $x\in S_\X$ and $\tau\le \taubar$ that 
		\begin{align*} 
			\delta_{\F,x}^{j_p(x)}(\tau)\ge \phi(\norm{x-y}),
		\end{align*}
		where $\phi\colon \R^+\to\R^+$ is nondecreasing and $\phi(\tau)>0$ for $\tau>0$, then $\X$ is uniformly convex.
	\end{enumerate}
\end{thm}

\begin{proof}
	\emph{Claim} \ref{thm:asy_xu_roach1}:  Note that $\F=f\circ\norm{\cdot}$, with $f(t)=\tfrac{1}{p}t^p$, which is convex, thus $\rho_{\f,\F(x)}^{1}$ is nondecreasing by Proposition \ref{prop:mods_nondecr}, so Proposition \ref{prop:chain_rule} gives
	\begin{align*}
		\rho_{\F,x}^{j_p(x)}(\tau)\le 
		\rho_{\norm{\cdot},x}^{j_p(x)}(\tau)+\rho_{\f,\F(x)}^{1}\paren{\tau+\rho_{\norm{\cdot},x}^{j_p(x)}(\tau)}.
	\end{align*}
	We have by Taylor's theorem 
	\[\rho_{\f,1}^{1}(\tau)=\sup_{\sigma\in\{-1,+1\}}\frac{p-1}{2}\tau^2+ r(\sigma\tau)\tau^2\le C \tau^2, \text{ for } \tau\le 3\taubar
	,\] 
	where the second inequality holds as $\rho_{\f,1}^{1}$ is always finite and so is the remainder $r$.
	We have $j_p(x)\in\partial\norm{\cdot}(x)$ for $x\in S_\X$, so by Proposition \ref{prop:smoo_mods_equiv} we have 
	$\rho_{\norm{\cdot},x}^{j_p(x)}(\tau)\le 2\rho_\X(\tau)$   and one can easily see that $\rho_\X(\tau)\le \tau$. So we have
	\begin{align*}
		\rho_{\F,x}^{j_p(x)}(\tau)\le 2\rho_\X(\tau)+9C\tau^2\le (2+9C/C_\tau)\rho_\X(\tau), \qquad \tau\le \taubar
	\end{align*}
	where the second inequality follows from Lemma \ref{lem:mods_properties}, \ref{lem:mods_properties4}.
	
	\emph{Claim} \ref{thm:asy_xu_roach2}: Note that $\norm{\cdot}=f^{-1}\circ\F$ and $f^{-1}(t)=\paren{pt}^\frac{1}{p}$ is concave, thus $-f^{-1}$ is convex and it is differentiable,
	so $-1\in\partial\paren{-f^{-1}}\paren{\tfrac{1}{p}}$ and by Proposition \ref{prop:mods_nondecr} $\rho_{\f^{-1},1/p}^{1}=\rho_{-\f^{-1},1/p}^{-1}$ is nondecreasing.
	Then Proposition \ref{prop:chain_rule} gives
	for all $x\in S_\X$ that
	\begin{align*}
		\rho_{\norm{\cdot},x}^{j_p(x)}(\tau)\le \rho_{\F,x}^{j_p(x)}(\tau)+ \rho_{\f^{-1},1/p}^{1}\paren{\tau+\rho_{\F,x}^{j_p(x)}(\tau)}\le \phi(\tau)+C_{\taubar}\tau^2,
	\end{align*}
	where the second inequality follows by Taylors theorem as above and the fact that by Claim \ref{thm:asy_xu_roach1} we always have $\rho_{\F,x}^{j_p(x)}(\tau)\le C\tau$ for some $C>0$.
	Thus Proposition \ref{prop:smoo_mods_equiv} gives the claim.
	
	 \emph{Claim} \ref{thm:asy_xu_roach3}: First of all note that $\F^*(t)=\tfrac{1}{p'}t^{p'}$, with $\tfrac{1}{p}+\tfrac{1}{p'}=1$. We have  
	\begin{align*}
		\delta_{\F,x}^{j_p(x)}(\tau)\ge \paren{\delta_{\F,x}^{j_p(x)}}^{**}(\tau)=\paren{\rho_{\F^*,j_p(x)}^{x}}^*(\tau)
		=\sup_{r\ge 0}\left[\tau r- \rho_{\F^*,j_p(x)}^{x}(r)\right].
	\end{align*}
	By Claim \ref{thm:asy_xu_roach1} we have for all $x\in S_\X$ that  $\rho_{\F^*,j_p(x)}^{x}(r)\le C_{\taubar,p'}\rho_{\X^*}(r)$ for all $0<r<\taubar$. 
	We are only interested in the case $\tau\to 0$ so let $\tau\le C_{\taubar,p'}\rho_{\X^*}(\taubar)/\taubar$, where $\rho_{\X^*}(\taubar)/\taubar >0$ by Lemma \ref{lem:mods_properties}, \ref{lem:mods_properties4}.
	Then by  Lemma \ref{lem:mods_properties}, \ref{lem:mods_properties1}. we 
	have  $\tau r\le C_{\taubar,p'}\rho_{\X^*}(r)$ for $r\ge \taubar$ and thus find
	\begin{align*}
		\sup_{0\le r}\left[\tau r- \rho_{\F^*,j_p(x)}^{x}(r)\right]\ge \sup_{0\le r\le \taubar}\left[\tau r- C_{\taubar,p'} \rho_{\X^*}(r)\right]=\paren{C\rho_{\X^*}}^*(\tau).
	\end{align*}	
	So we have by Lemma \ref{lem:mods_properties}, \ref{lem:mods_properties2} and \ref{lem:mods_properties3}, that
	\begin{align*}
		\delta_{\F,x}^{j_p(x)}(\tau)\ge\paren{C_{\taubar,p'}\rho_{\X}}^*(\tau)=\frac{C_{\taubar,p'}}{2}\paren{2\delta_{\X}}^{**}\paren{\frac{2\tau}{C_{\taubar,p'}}}
		\ge C_{\taubar,p'}\paren{\delta_{\X}}\paren{\frac{\tau}{C_{\taubar,p'}}}.
	\end{align*}
	
	\emph{Claim} \ref{thm:asy_xu_roach5}:  By assumption we have by $\delta_{\F,x}^{j_p(x)}(\tau)\ge \phi(\tau)$ for $\tau\le \taubar$ and by Proposition \ref{prop:mods_nondecr} we
	have for $\tau>\taubar $ that  $ \delta_{\F,x}^{j_p(x)}(\tau)\ge \tau\delta_{\F,x}^{j_p(x)}(\taubar)/\taubar$ and thus $\delta_{\F,x}^{j_p(x)}(\tau)\ge\tilde{\phi}(\tau)$ with
	\begin{align*}
		\tilde{\phi}(\tau):=\begin{cases}
						\phi(\tau), &\tau \le \taubar,\\
						\tau\phi(\taubar)/\taubar, &\tau > \taubar .	
					\end{cases}
	\end{align*}
	So by Proposition \ref{prop:conv*=smoo} we have
	for all $x^*\in S_{\X^*}$ that
	\begin{align*}
		\rho_{\F^*,x^*}^{j_p^*(x^*)}(\tau)=\paren{\delta_{\F,j_p^*(x^*)}^{x^*}}^*(\tau)\le \tilde{\phi}^*(\tau).
	\end{align*}
	Now just observe that for $\tau<\phi(\taubar)/\taubar$ we have
	\begin{align*}
		\frac{\tilde{\phi}^*(\tau)}{\tau}=\sup_{0\le t}\left[ t-\frac{\tilde{\phi}(t)}{\tau}\right]=\sup_{0\le t\le \taubar}\left[ t-\frac{\phi(t)}{\tau}\right]\to 0, \,\tau\to 0,
	\end{align*}
	as $\phi$ is nondecreasing. So by part \ref{thm:asy_xu_roach2} of the theorem we get that $\X^*$ is uniformly smooth from which it follows that $\X$ is uniformly convex \cite[Prop. 1.e.2]{Lindenstrauss1979}.
\end{proof} 

\begin{rem}
	One can see from the above proof, that in the asymptotic case $\taubar\to 0$ one can choose the constant $C_{\taubar,p}$ such that
	\begin{align*}
		C_{\taubar,p}\to 
		\begin{cases}
			2, &\X \text{ is not 2-smooth} \\
			1+p, &\X \text{ is 2-smooth}.
		\end{cases}
	\end{align*}
	These constants are not sharp for every space $\X$, but atleast in the asymptotic case the constants are much simpler than the ones given in \cite{Xu1991}. For best known constants with respect to $L^p$ spaces we refer to \cite{Xu89} and \cite{Xu94}.
\end{rem}
The above theorem combined with Proposition \ref{prop:homo} gives us upper and lower bounds on the Bregman divergence for $\norm{x-y}\le\taubar \norm{x}$. However as for large $\norm{x-y}$ the Bregman divergence will be dominated by the term $\norm{y}^p$ it is not difficult to also find bounds that hold for all $x,y\in X$. Further one can also easily conclude bounds for the symmetric Bregman divergence,
\begin{align*}
	\symbreg{\F}(x,y):=\breg{\F}{j_p(x)}(y,x)+\breg{\F}{j_p(y)}(x,y)=\lsp j_p(x)-j_p(y),x-y\rsp,
\end{align*}
from our theorem. These two claims are shown in the following two propositions.
\begin{prop} \label{prop:uni_sym_upper}
	For some fixed $p>1$ let $\F=\frac{1}{p}\norm{\cdot}^p$ and let $\phi \colon  \R^+\to\R^+$ be nondecreasing. Let $V=\X\setminus \{0\} \times \X$ and define the statements:
		\begin{align}
			\label{eq:unsymmetric_bound} \tag{a}	&\exists C,c>0 \forall (x,y)\in V, \norm{x-y}\le c\norm{x} :  \breg{\F}{j_p(x)}(y,x)\le C\norm{x}^p\phi\paren{\tfrac{\norm{x-y}}{\norm{x}}} \\
			\label{eq:symmetric_bound} \tag{b}	&\exists C>0 \forall (x,y)\in V :  \symbreg{\F}(x,y)\le C \max\{\norm{x},\norm{y}\}^p \phi\paren{\tfrac{2\norm{x-y}}{\max\{\norm{x},\norm{y}\}}} \\
			\label{eq:uniform_bound} \tag{c} &\exists C>0 \forall (x,y)\in V :  \breg{\F}{j_p(x)}(y,x)\le C \max\{\norm{x},\norm{y}\}^p \phi\paren{\tfrac{2\norm{x-y}}{\max\{\norm{x},\norm{y}\}}}
		\end{align}
	Then \eqref{eq:unsymmetric_bound} $\Rightarrow$  \eqref{eq:symmetric_bound} $\Rightarrow$ \eqref{eq:uniform_bound}. 
	Obviously one also has \eqref{eq:uniform_bound} $\Rightarrow$ \{\eqref{eq:unsymmetric_bound} with $\phi$ replaced by $\phi(2\cdot)$\}.
\end{prop}
\begin{proof}	
	We only show that \eqref{eq:unsymmetric_bound} implies \eqref{eq:symmetric_bound} as \eqref{eq:symmetric_bound} $\Rightarrow$ \eqref{eq:uniform_bound}
	follows trivially. Without loss of generality let $c\le 1$. First of all assume $\frac{\norm{x-y}}{\norm{x}}>c$. Then by 
	\begin{align*}
		\frac{\norm{x-y}}{\norm{x}}\frac{\norm{x}}{\norm{y}}=\frac{\norm{x-y}}{\norm{y}}\ge \frac{\norm{y}-\norm{x}}{\norm{y}} \ge 1- \frac{\norm{x}}{\norm{y}}
	\end{align*}
	one can see that no matter if we have $\norm{x}/\norm{y}>1/2$ or$\norm{x}/\norm{y}\le 1/2$ one always has $\frac{2\norm{x-y}}{\norm{y}}>c$. So by 
	\begin{align*}
		\symbreg{\F}(x,y)=\lsp j_p(x)-j_p(y),x-y\rsp &\le \norm{x}^p+\norm{y}^p+\norm{x}^{p-1}\norm{y}+\norm{y}^{p-1}\norm{x} \\
		 &\le 4\max\{\norm{x},\norm{y}\}^p
	\end{align*}
	we find that \[\symbreg{\F}(x,y)\le \frac{4}{\phi\paren{c}}\max\{\norm{x},\norm{y}\}^p \phi\paren{\frac{2\norm{x-y}}{\max\{\norm{x},\norm{y}\}}}.\]
	
	Now consider the case $\norm{x-y}/\norm{x}\le c\le 1$. We can conclude that $\norm{y}\le 2\norm{x}$, so that 
	\[\phi\paren{\frac{\norm{x-y}}{\norm{x}}}\le \phi\paren{\frac{2\norm{x-y}}{\norm{y}}}\] so by $\eqref{eq:unsymmetric_bound}$ we see that \eqref{eq:symmetric_bound} holds true.
\end{proof}

\begin{prop} \label{prop:uni_sym_lower}
	For some fixed $p>1$ let $\F=\frac{1}{p}\norm{\cdot}^p$ and let $\phi\colon \R^+\to\R^+$ be nondecreasing and $\phi(\tau)>0$ for $\tau>0$. 
	Let $V=\X\setminus \{0\} \times \X$ and define the statements:
		\begin{align}
			\label{eq:unsymmetric_lowerbound} \tag{d}	&\exists C,c>0 \forall (x,y)\in V, \norm{x-y}\le c\norm{x} :  \breg{\F}{j_p(x)}(y,x)\ge C\norm{x}^p\phi\paren{\tfrac{\norm{x-y}}{\norm{x}}} \\
			\label{eq:uniform_lowerbound} \tag{e}	&\exists C>0 \forall(x,y)\in V :  \breg{\F}{j_p(x)}(y,x)\ge C \max\{\norm{x},\norm{y}\}^p \phi\paren{\tfrac{\norm{x-y}}{\max\{\norm{x},\norm{y}\}}} \\
			\label{eq:symmetric_lowerbound} \tag{f}	&\exists C>0 \forall (x,y)\in V :  \symbreg{\F}(x,y)\ge C \max\{\norm{x},\norm{y}\}^p \phi\paren{\tfrac{\norm{x-y}}{\max\{\norm{x},\norm{y}\}}} 
		\end{align}
		Then \eqref{eq:unsymmetric_lowerbound} $\Rightarrow$ \eqref{eq:uniform_lowerbound} $\Rightarrow$ \eqref{eq:symmetric_lowerbound}.
\end{prop}
\begin{proof}
	The proof is very similar to the previous proof so we just sketch it.  We look at three different cases.
	By  Proposition \ref{prop:mods_nondecr} we know that  $\delta_{\F,x}^{j_p(x)}$ is nondecreasing, so \eqref{eq:unsymmetric_lowerbound} gives also for 
	$\norm{x-y}/\norm{x}\ge c$ that $\breg{\F}{j_p(x)}(y,x)\ge C\norm{x}^p\phi(c)$ and thus
	\begin{align*}
		 \breg{\F}{j_p(x)}(y,x)\ge\begin{cases}
						C\norm{x}^p\phi\paren{\frac{\norm{x-y}}{\norm{x}}}, & \frac{\norm{x-y}}{\norm{x}} \le c,\\
						C\norm{x}^p\phi(c), &c \le \frac{\norm{x-y}}{\norm{x}}<		N, \\
						C_{p,\phi,N}\norm{y}^p\phi\paren{\frac{\norm{x-y}}{\norm{y}}}, &N \le \frac{\norm{x-y}}{\norm{x}},		
					\end{cases}
	\end{align*}
	for sufficiently large $N>3$, where the last line follows from the definition of the Bregman divergence and the fact that $\frac{\norm{x-y}}{\norm{x}}\to \infty$ implies $\norm{y}\to \infty$ implies $\frac{\norm{x-y}}{\norm{y}}\to 1$. To conclude \eqref{eq:uniform_lowerbound} one then basically has to redefine the constants. \eqref{eq:symmetric_lowerbound} follows trivially.
\end{proof}

To conclude this chapter we combine the results and summarize the most important inequalities.
\begin{cor}	Let $\X$ be a Banach space and $\F(x)=\tfrac{1}{p}\norm{x}^p$ for $p>1$ then there exists constants
$C_1,C_2>0$ such that for all $x,y\in\X$ we have
\begin{align}
	\breg{\F}{j_p(x)}(y,x)\le C_1 \max\{\norm{x},\norm{y}\}^p \rho_\X\paren{\frac{2\norm{x-y}}{\max\{\norm{x},\norm{y}\}}}
\end{align}
and
\begin{align}\label{eq:final_lower_bound}
	\breg{\F}{j_p(x)}(y,x)\ge C_2 \max\{\norm{x},\norm{y}\}^p \delta_\X\paren{\frac{\norm{x-y}}{3\max\{\norm{x},\norm{y}\}}}.
\end{align}
If the space $\X$ is $\smoo$-smooth, then there exists $C>0$ and  for all $\taubar>0$ also $C_{\taubar}>0$ such that
\begin{align}
\begin{split}
	\breg{\F}{j_p(x)}(y,x)\le 
	\begin{cases}	
	C\norm{x-y}^\smoo , & p=\smoo \\
	C_{\taubar}\norm{x}^{p-\smoo}\norm{x-y}^\smoo,\text{ for }\frac{\norm{x-y}}{\norm{x}}\le \taubar, & p\neq \smoo .
	\end{cases}
	\end{split}
\end{align}
If the space $\X$ is $\conv$-convex, then there exists $\tilde{C}>0$ and for all $\taubar>0$ also $\tilde{C}_{\taubar}>0$ such that
\begin{align}
\begin{split}
	\breg{\F}{j_p(x)}(y,x)\ge 
	\begin{cases}	
	\tilde{C}\norm{x-y}^\conv , & p=\conv \\
	\tilde{C}_{\taubar}\norm{x}^{p-\conv}\norm{x-y}^\conv,\text{ for }\frac{\norm{x-y}}{\norm{x}}\le \taubar, & p\neq \conv .
	\end{cases}	
		\end{split}
\end{align}
\end{cor}
\begin{proof}
	Theorem \ref{thm:asy_xu_roach} shows the bounds for $x\in S_\X, \norm{x-y}\le \taubar$, Proposition \ref{prop:homo} then gives the bounds for all $x\in\X$ and $\norm{x-y}\le \taubar\norm{x}$. Apply Proposition \ref{prop:uni_sym_upper} and Proposition \ref{prop:uni_sym_lower} to get the bounds for all $x,y\in\X$. 
\end{proof}

%
%


\begin{acknowledgements}
	I thank my supervisor Thorsten Hohage for many helpful comments. Financial support by Deutsche Forschungsgemeinschaft through grant CRC 755, project C09, and RTG 2088 is gratefully acknowledged. 
\end{acknowledgements}

\bibliography{literature}   

\end{document}